\documentclass[reqno,12pt,a4letter]{amsart}
\usepackage{amsmath, amsxtra, amssymb, latexsym, amscd, amsthm}
\usepackage{graphicx, color}
\usepackage[active]{srcltx}
\usepackage[utf8]{inputenc}
\usepackage[mathscr]{euscript}
\usepackage{mathrsfs,cite}
\usepackage[english]{babel}
\usepackage{color,enumerate,comment}
\usepackage[active]{srcltx}
\newtheorem{theorem}{Theorem}[section]
\newtheorem{corollary}[theorem]{Corollary}
\newtheorem{lemma}[theorem]{Lemma}
\newtheorem{proposition}[theorem]{Proposition}
\theoremstyle{definition}
\newtheorem{definition}[theorem]{Definition}
\newtheorem{example}[theorem]{Example}
\newtheorem{remark}[theorem]{Remark}

\newcommand{\Limsup}{\mathop{{\rm Lim}\,{\rm sup}}}
\numberwithin{equation}{section}

\title[Second-order optimality conditions for MOP]{Second-order optimality conditions for multiobjective optimization problems with constraints}

\author[N.Q. Huy]{Nguyen Quang Huy}
\address[N.Q. Huy]{Department of Mathematics, Hanoi Pedagogical University 2, Xuan Hoa, Phuc Yen, Vinh Phuc, Vietnam}
\email{\tt huyngq308@gmail.com; nqhuy@hpu2.edu.vn}

\author[B.T. Kien]{Bui Trong Kien}

\address[B.T. Kien]{Department of Optimization and Control Theory, Institute of Mathematics, VAST, 18
Hoang Quoc Viet Road, Hanoi, Vietnam }
\email{\tt btkien@math.ac.vn}

\author[G.M. Lee]{Gue Myung Lee}

\address[G.M. Lee]{Department of Applied Mathematics, Pukyong National University, Busan 48513, Korea}
\email{\tt gmlee@pknu.ac.kr}

\author[N.V. Tuyen]{Nguyen Van Tuyen$^*$}
\address[N.V. Tuyen]{Department of Mathematics, Hanoi Pedagogical University 2, Xuan Hoa, Phuc Yen, Vinh Phuc, Vietnam}
\email{\tt tuyensp2@yahoo.com; nguyenvantuyen83@hpu2.edu.vn}


\thanks{$^*$Corresponding.}
\thanks{The research of Nguyen Van Tuyen was supported by the Ministry of Education and Training of Vietnam [grant number B2018-SP2-14]}
\date{\today}

\begin{document}

\begin{abstract}
In this paper, we introduce  the second-order subdifferentials for functions which are  G\^ateaux differentiable on an open set and whose G\^ateaux derivative mapping is locally Lipschitz.  Based on properties of this
kind of second-order subdifferentials and techniques of variational
analysis, we derive second-order necessary conditions for  weak Pareto efficient solutions of multiobjective  programming problems with constraints.
\end{abstract}

\maketitle


\section{Introduction}

Let $X, Y$, and $Z$ be Banach spaces with the dual spaces $X^*, Y^*,$
and $Z^*$, respectively. Throughout the paper we assume that the
unit ball $B^*\subset X^*$ is weak$^*$-sequentially compact. Let $D$
be a nonempty open subset in $X$ and  $Q$ be a closed convex set in
$Z$ with nonempty interior. Given mappings  $f_j\colon D\to
\mathbb{R}$,  $H\colon D\to Y$ and $G:
D\to Z$, we consider the following constrained multiobjective
programming problem
\begin{equation}\notag
(P)\quad\quad
\begin{cases}
 \text{Min}_{\mathbb{R}^m_+} F(x):=(f_1(x), \ldots, f_m(x))\\
\text{subject to}\\
 H(x)= 0,\\
 G(x)\in Q.
 \end{cases}
 \end{equation}
 The prototype of such problem arises in control theory with state equations and pointwise constraints. The goal of this paper is to derive second-order
necessary conditions for  problem $(P)$ in term of a notion of
second-order subdifferentials for functions which are of class
$C^{1,1}(D)$. Recall that a function $\phi \colon D\to \mathbb{R}$
is said to be of class $C^{1, 1}(D)$ if its first-order G\^ateaux
derivative $\phi'\colon D\to X^*$ is locally Lipschitz on $D$.

By introducing generalized second-order directional derivatives and using
techniques of variational analysis, P\'ales and Zeidan
\cite{Pales1} gave second-order necessary conditions for mathematical programming  problems in the form $(P)$, i.e., when $m=1$, and some problems where the objective function is the maximum of smooth
functions depending on a parameter from a compact metric space.
To our knowledge, this result has been the best one on the second-order
necessary conditions so far. Instead of generalized second-order
directional derivatives,  Georgiev and Zlateva \cite{Pando}
introduced the so-called second-order Clarke subdifferentials for functions
of class $C^{1, 1}(D)$ in the case, where the dual $X^*$ is
separable. This definition is based on a result of Christensen
related to the almost everywhere differentiability of Lipschitz
functions $\psi\colon X\to Z$ with $Z$ is a Banach space which has a
Radon--Nikodym property. Then the authors obtained second-order
necessary conditions and sufficiently conditions for mathematical programming  problems in the form $(P)$ in terms of second-order Clarke subdifferentials.  

The study of second-order optimality conditions for vector
optimization problems is of the concern of some mathematicians. For the papers which have close connection to the present
work, we refer the readers to \cite{Gfrerer,Khan13,Khan17} and references therein. Let us give briefly  some comments on the considered problems and the obtained results of those papers. In \cite{Gfrerer}, under the Robinson
qualification constraint conditions, the author derived second-order
necessary optimality conditions and sufficient optimality
conditions for vector optimization problems where the mappings are
second-order directionally differentiable.  By using the Dubovitskii--Milyutin approach, the authors  \cite{Khan13,Khan17} obtained some second-order necessary optimality conditions in terms of second-order tangential derivatives for set-valued optimization problems. For more discussions on the recent development of the second-order derivatives relative to optimal
conditions in nonsmooth analysis, the reader is referred to  \cite{Bonnans99,Borwein,Chieu-17,Kien-18,Kim-18,Mord1,Mord2,Pales2,Yang92,Huy16,Huy17,Kawasaki88,Pales1994,Tuyen18} and the references therein.

In this paper we derive the second-order necessary
optimality conditions for problem $(P)$, where the Robinson
qualification constraint conditions may not be valid and  the
mappings may not be second-order differentiable. To do this, we first introduce second-order subdifferentials for functions of class $C^{1,1}(D)$ and give some properties for this  kind of second-order
subdifferentials. We then  utilize the Dubovitskii--Milyutin approach as
well as techniques of variational analysis of \cite{Pales1} to deal
with the problem. The obtained results improve and generalize the corresponding results of \cite[Theorem 2.4]{Pando}, \cite[Theorem 6]{Pales1} and \cite[Theorem 8.2]{Pales1994}. We also show that our results still hold  for critical directions which may not be regular.   

The rest of our paper consists of two sections. In Section \ref{Sect2}, we
present some properties of second-order subdifferentals and some
results related to variation sets of second-order. Section \ref{Sect3} is destined for first- and second-order necessary conditions for weak Pareto efficient solutions of $(P)$. 
\section{Second-order subdifferentials and second-order variations}
\label{Sect2}
\subsection{Second-order subdifferentials}
Let  $f\colon D\to \mathbb{R}$  be a locally Lipschitz function on $D$. Recall that the Clarke subdifferential of $f$ at $\hat x\in D$ is defined by
$$
\partial f(\hat x):=\{x^*\in X^*\;|\; \langle x^*, h\rangle \leq f^\circ(\hat x; h),\ \ \forall h\in
X\},
$$ where
$$
f^\circ (\hat x; h):=\limsup_{x\to\hat x, \atop{t\to 0^+}}\frac{f(x+ th)-
    f(x)}{t}
$$ is the Clarke directional derivative of $f$ at $\hat
x$ in the direction $h$.

Denote by ${\mathcal L}(X\times X)$ the Banach space of all
bilinear continuous functionals $L\colon X\times X\to \mathbb{R}$ with the norm
$$\|L\|:=\sup_{\|h_1\|=1\atop {\|h_2\|=1}}|L(h_1, h_2)|,
$$
and ${\mathcal L}(X, X^*)$ the Banach space of all linear continuous
mappings $L\colon X\to X^*$ with the norm $$\|L\|:=\displaystyle\sup_{\|h\|=1}\|L(h)\|_*.$$
It is well-known that ${\mathcal L}(X\times X)$ and  ${\mathcal L}(X, X^*)$ are
isometrically isomorphic; see \cite[Section 2.2.5]{Alek}. So, in the
sequel, we identify ${\mathcal L}(X\times X)$ and  ${\mathcal L}(X,
X^*)$. By this way, if $f\colon D\to \mathbb{R}$ is twice G\^ateaux differentiable
at $\hat x\in D$, then $f''(\hat x)$ is a linear mapping from $X$ to
$X^*$. This suggests us to introduce the following definition.

\begin{definition}\label{def-second-order-subdiff}{\rm  Let $f\in C^{1, 1}(D)$ and $\hat x\in D$. The {\em second-order subdifferential} of $f$ at $\hat x$ is the set-valued   map $$\partial^2 f(\hat x)\colon X\rightrightarrows X^*,$$ which is     defined by
    $$\partial^2 f(\hat x)(d):=\partial\langle f'(\,\cdot\,), d\rangle (\hat x), \ \ \forall \ d\in X.$$
}
\end{definition}
Note that, by the Hahn--Banach Theorem, $\partial^2 f(\hat
x)(d)$ is always nonempty for all $d\in X$.

\medskip
The following proposition summarizes some properties of $
\partial^2 f(\,\cdot\,)$.

\begin{proposition}\label{prop2.1} Suppose that  $f$ and $g$ are of class $C^{1, 1}(D)$.  Then the following assertions hold:
    \begin{enumerate}[\rm (i)]
        \item The mapping $\partial^2 f(\hat x):
        X\rightrightarrows X^*$  has nonempty convex and $w^*$-compact
        valued.

        \item  For each $d\in X$, the mapping $x\mapsto
        \partial^2 f(x)(d)$  from  $(D, \|\cdot\|)$ to $X^*$ is local bounded and upper
        semicontinuous at $\hat x$, that is, if  $x_n\to \hat x$,
        $L_n\xrightarrow{w^*}{L}$ and $L_n\in\partial^2 f(x_n)(d)$, then
        $L\in \partial^2 f(\hat x)(d)$.

        \item   If $f$ is twice continuously G\^ateaux differentiable
        at $\hat x\in D$, then  $\partial^2 f(\hat x)=\{f''(\hat x)\}$.

        \item   For any $d\in X$ and $s\in \mathbb{R}$, one has
        \begin{align*}
        \partial^2 f(\hat x)(sd)&= s\partial^2 f(\hat x)(d);
        \\
        \partial^2 (f + g)(\hat x)(d)&\subset \partial^2
        f(\hat x)(d) +\partial^2 g(\hat x)(d).
        \end{align*}
    \end{enumerate}
\end{proposition}
\begin{proof} (i) Since $f'(\,\cdot\,)$ is locally Lipschitz
    around $\hat x$ with constant $l>0$, the mapping
    $$x\mapsto \langle f'(x), d\rangle$$ is locally Lipschitz around
    $\hat x$ with constant $l\|d\|$.  Hence,
    $$\|L\|\leq l\|d\|, \ \ \forall L\in \partial^2 f(\hat x)(d).$$ By the
    Banach--Alaoglu--Bourbaki Theorem, $\partial^2 f(\hat x)(d)$ is a
    $w^*$-compact set. The convexity of $\partial^2 f(\hat x)(d)$ is easy to check, so omitted.

    \medskip
    \noindent (ii) The assertion follows directly from
    \cite[Proposition 2.1.5]{Clarke2}.

    \medskip
    \noindent (iii) Fix $d\in X$, we have
    \begin{align}
    \partial \langle f'(\,\cdot\,), d\rangle(\hat x) &=\bigg\{L\in X^*\,|\, \langle L,
    h\rangle \leq \limsup_{x\to\hat x\atop{\varepsilon\to 0^+}}\frac{ \langle f'(x + \varepsilon h)- f'(x),
        d\rangle}{\varepsilon},\ \ \forall h\in X\bigg\}\notag\\
    &=\{ L\,|\, \langle L, h\rangle \leq f''(\hat x)(d, h),\ \ \forall h\in
    X\}=\{L\,|\, L= f''(\hat x)(d)\}\notag.
    \end{align}
    Hence, $\partial^2 f(\hat x)(d)= f''(\hat x)(d)$ for all $d\in X.$

\medskip
    \noindent (iv) By \cite[Proposition 2.3.1]{Clarke2}, we have
    \begin{align}
    \partial^2 f(\hat x)(sd)&=\partial\langle f'(\,\cdot\,), (sd)\rangle(\hat x)=\partial(s\langle f'(\,\cdot\,),
    d\rangle)(\hat x)\notag\\
    &=s\partial\langle f'(\,\cdot\,), d\rangle(\hat x)= s\partial^2 f(\hat
    x)(d).\notag
    \end{align}
    The second assertion follows directly from \cite[Proposition 2.3.3]{Clarke2}.
\end{proof}

\bigskip
To illustrate how to compute $\partial^2 f(\hat x)$ we give a simple
example for the case where $X=\mathbb{R}^2$.

\begin{example}\rm  Let $X=\mathbb{R}^2$ and $f(x, y)=\int_0^x |s|ds + y^2$.
    Then
    \begin{equation}\notag
    \partial^2 f(0, 0)= \Big\{ \Big(
    \begin{array}{cc}
    a & 0 \\
    0 & 2 \\
    \end{array}
    \Big)\;|\; a\in [-1, 1]\Big\}.
    \end{equation}
    In fact, we have $f'(x,y)=(|x|, 2y)$. Hence, for any $d=(d_1,
    d_2)$, one has $$\langle f'(x, y), d\rangle= d_1|x|+ 2d_2 y.$$ It
    follows that
    \begin{align}\notag
    \partial^2 f(0, 0)(d)&= (d_1\partial(|x|)|_{x=0},  2d_2\partial (y)|_{y=0})=\{(ad_1, 2d_2)\;|\; a\in [-1, 1]\}.
    \end{align}Hence
    \begin{equation}\notag
    \partial^2 f(0, 0)(d)=\Big\{\Big(
    \begin{array}{cc}
    a & 0 \\
    0 & 2 \\
    \end{array}\Big) \Big(
    \begin{array}{c}
    d_1\\
    d_2 \\
    \end{array}\Big)\,|\, a\in[-1,1]\Big\}.
    \end{equation} We obtain the desired formula.
\end{example}

The following mean value theorem plays an important role in our
paper.

\begin{theorem}\label{mean-th}  Let $f\in C^{1,1}(D)$. Then, for every $a,b\in D$ with $[a,b]\subset D$,
    there exist $\xi\in (a,b)$ and $L\in \partial^2 f(\xi)(b-a)$
    such that
    $$f(b)-f(a)-\langle f'(a), b-a\rangle =\frac{1}{2}\langle L,  b-a\rangle.$$
\end{theorem}
\begin{proof} For the proof we need the following lemma.
    \begin{lemma}[{see \cite[Proposition 1.14]{Pando}}] Let $I$ be an open
        interval containing $[0, 1]$ and $\phi\in C^{1, 1}(I)$. Then, there
        exits $t_0\in (0, 1)$ such that
        \begin{equation}\label{1}
        \phi(1)-\phi(0) -\phi'(0)\in \frac{1}2\partial \phi'(t_0).
        \end{equation}
    \end{lemma}
    We  now  define a function $\phi(t):= f(a +th)$, $t\in [0, 1]$ with
    $h:=b-a$. It is clear that $\phi$ satisfies properties of the above
    lemma. Therefore, there exists $t_0\in (0, 1)$ such that
    \eqref{1} is satisfied. Since $\phi'(t)= \langle f'(a+th),
    h\rangle$, the chain rule (see \cite[Theorem 2.3.10]{Clarke2})
    implies that
    \begin{align}\notag
    \partial \phi'(t_0)= \partial \langle f'(\,\cdot\,),
    h\rangle(a+t_0h)(h)=\partial \langle f'(\,\cdot\,),h\rangle(\xi)(h)=\partial^2f(\xi)(h)(h),
    \end{align}where $\xi= a+ t_0(b-a) \in (a, b)$. Hence, there exists
    $L\in \partial^2f(\xi)(b-a)$ such that
    \begin{align*}
    f(b)-f(a)-\langle f'(a), b-a\rangle=\phi(1)-\phi(0)
    -\phi'(0)=\frac{1}2\langle L, b-a\rangle.
    \end{align*}
We obtained the desired conclusion of Theorem 2.1.
\end{proof}
Let $Y$ be a Banach space and $H\colon D\to Y$ be a mapping defined on $D$. We say that $H$ is {\it strictly Fr\'echet differentiable} at $\hat x\in D$,
if there exists a linear continuous mapping $H'(\hat x)\colon X\to Y$ such
that for all $\varepsilon>0$, there exists $\delta>0$ with
$$\|H(u)-H(v)-\langle H'(\hat x),u-v\rangle\|\leq \varepsilon\| u-v\|
$$
whenever $u$ and $v$ satisfy $\|u-\hat x\|<\delta$ and $\|v-\hat x\|<\delta$. It is easy to see that
$$\langle H'(\hat x),d\rangle=\lim_{x\to \hat x\atop{\varepsilon\to 0^+}}\frac{H(x+\varepsilon d)-H(x)}{\varepsilon}
$$
holds for all $d\in X$.

According to \cite{Pales1}, when  $H$ is strictly Fr\'echet
differentiable at $\hat x$ and $d\in X$, then the {\em second-order weak directional derivative}  of $H$ at $\hat x$ in the direction $d$  is defined by
$$H''(\hat x; d):=\bigg\{y\in Y\,|\, \liminf_{\varepsilon\to 0^+}\Big\| y-2\frac{H(\hat x+\varepsilon d)-H(\hat x) -\varepsilon \langle H'(\hat x), d\rangle}{\varepsilon^2} \Big\|=0\bigg\}.
$$ In other words, using the concept of the sequential
Painlev\'e--Kuratowski upper limit of \cite{Au}, we have
$$
H''(\hat x; d)=\Limsup_{\varepsilon\to 0^+}\Big[ 2\frac{H(\hat x+\varepsilon d)-H(\hat x) -\varepsilon \langle H'(\hat x), d\rangle}{\varepsilon^2}\Big].
$$
This set may be empty. If it is nonempty, then we say that $H$ is twice weakly directionally differentiable at $\hat x$ in the direction $d$.  It is clear that when $H$ is of class $C^2$, then
$$H''(\hat x; d)=\{ H''(\hat x)(d)(d)\}.$$

Now we compare the second-order weak directional derivative with the second-order subdifferential in the sense of Definition \ref{def-second-order-subdiff}.
\begin{proposition}\label{relative-second-order-derivatives} Let $H$ be a real-valued function defined on $D$ and $\hat x\in D$. If $H\in C^{1,1}(D)$,  then $H$ is twice weakly directionally differentiable at $\hat x$ in the any direction $d\in X$ and $H''(\hat x; d)\subset \partial^2 H(\hat x)(d)(d)$.
\end{proposition}
\begin{proof} Let $d\in X$ and $\varepsilon_n$ be an arbitrary positive sequence converging to $0$ as $n\to \infty$.   For each $n\in\mathbb{N}$, by Theorem \ref{mean-th}, there exist $t_n\in(0,1)$  and $L_n\in \partial^2H(\hat x+t_n\varepsilon_nd)(d)$ such that
    $$H(\hat x +\varepsilon_n d)- H(\hat x)- \varepsilon_n \langle H'(\hat x), d\rangle=\frac{1}2 \varepsilon_n^2\langle L_n,
    d\rangle.
    $$
    By Proposition \ref{prop2.1}, we can
    assume that $L_n$ converges weakly$^*$ to $L\in \partial^2 f(\hat x)(d)$.  This implies that
    $$\lim\limits_{n\to\infty}2\dfrac{H(\hat x +\varepsilon_n d)- H(\hat x)- \varepsilon_n \langle H'(\hat x), d\rangle}{\varepsilon^2_n}=\langle L, d\rangle.$$
    Thus, $\langle L, d\rangle\in H''(\hat x; d)$ and $H''(\hat x; d)$ is nonempty.

    To prove the second assertion, fix $y\in H''(\hat x; d)$. Then there exists a positive sequence $\varepsilon_n$ converging to $0$ such that
    $$\lim\limits_{n\to\infty}2\dfrac{H(\hat x +\varepsilon_n d)- H(\hat x)- \varepsilon_n \langle H'(\hat x), d\rangle}{\varepsilon^2_n}=y.$$
    For the sequence $\varepsilon_n$, as in the proof of the first assertion, there is  $L\in \partial^2 f(\hat x)(d)$ such that
    $$\lim\limits_{n\to\infty}2\dfrac{H(\hat x +\varepsilon_n d)- H(\hat x)- \varepsilon_n \langle H'(\hat x), d\rangle}{\varepsilon^2_n}=\langle L, d\rangle.$$
    This implies that $y=\langle L, d\rangle$ and we therefore get $H''(\hat x; d)\subset \partial^2 H(\hat x)(d)(d)$.
\end{proof}

The following result is immediate from the definition of the second-order weak directional derivative and Proposition \ref{relative-second-order-derivatives}.
\begin{corollary}\label{corollary27} Let $H:=(h_1, \ldots, h_p)\colon D\to \mathbb{R}^p$ be a vector-valued function  and $\hat x\in D$. If $h_i\in C^{1,1}(D)$ for all $i=1, \ldots, p$,  then $H$ is twice weakly directionally differentiable at $\hat x$ in the any direction $d\in X$ and
    $$H''(\hat x; d)\subset h_1''(\hat x; d)\times \ldots h_p''(\hat x; d) \subset  \partial^2 h_1(\hat x)(d)(d) \times\ldots\times \partial^2 h_p(\hat x)(d)(d).$$
\end{corollary}
\subsection{Second-order variations}
In this section, we recall some concepts related to second-order variations from \cite{Dubovitskii65,Pales1}.
\begin{definition}{\rm Let $f$ be a real-valued function defined on  $D$. A vector $\bar w\in X$ is called a {\em second-order descent
variation} of $f$ at $\hat x\in D$ in the direction $d$ if there exists an
$\bar \varepsilon>0$ such that $\hat x +\varepsilon d +\varepsilon^2(\bar w +
w)\in  D$ and
$$ f(\hat x +\varepsilon d +\varepsilon^2(\bar w +
w)) < f(\hat x)$$ for all $\varepsilon\in (0, \bar\varepsilon)$ and
$\|w\|<\bar\varepsilon$. The set of such $\bar w$ is denoted by
$W^2_{\delta}(f; \hat x, d)$. This set is always open.}
\end{definition}
Let $\Omega$ be a nonempty subset in $X$, $\hat x\in \Omega$ and $d\in X$.
\begin{definition}{\rm  A vector $\bar w\in X$ is said to be a {\em second-order
admissible variation} of $\Omega$ at $\hat x$ in the direction
$d$ if there exists an $\bar \varepsilon>0$ such that
$$\hat x +\varepsilon d +\varepsilon^2(\bar w + w)\in  \Omega
$$ for all $\varepsilon\in (0, \bar\varepsilon)$ and
$\|w\|<\bar\varepsilon$. We denote this set by
$W_{\alpha}^2(\Omega; \hat x, d)$, which is always open.}
\end{definition}

\begin{definition}{\rm The {\em second-order tangent variation set }  of  $\Omega$ at $\hat x$ in the direction $d$ is the set $W_{\tau}^2(\Omega; \hat x, d)$ of vectors $\bar w\in X$ such that there exist sequences  $\varepsilon_n\to 0^+$ and  $w_n\to 0$ satisfying
    $$\hat x +\varepsilon_n d +\varepsilon_n^2(\bar w + w_n)\in  \Omega \ \ \text{for all } \ \ n\in \mathbb{N}.
    $$
}
\end{definition}
\begin{remark}{\rm
\begin{enumerate} [{\rm (i)}]
    \item  Denote by $d_\Omega(x)$ the distance of $x$ from $\Omega$; then the set of all second-order tangent variations of $\Omega$ at $\hat x$ in the direction $d$ can be formulated as follows:
\begin{equation*}
W_{\tau}^2(\Omega; \hat x, d)=\Big\{\bar w\,|\, \liminf_{\varepsilon\to 0^+}\dfrac{d_\Omega(\hat x+\varepsilon d+\varepsilon^2\bar w)}{\varepsilon^2}=0\Big\}.
\end{equation*}
\item   It is easy to check that
\begin{equation*}
W_{\alpha}^2(X\setminus \Omega; \hat x, d)=X\setminus W_{\tau}^2(\Omega; \hat x, d).
\end{equation*}
\item   Suppose that $f$ is a real-valued function defined on $D$ and $\hat x\in D$. Then we have
\begin{equation*}
W^2_{\delta}(f; \hat x, d)=W^2_{\alpha}(\Omega; \hat x, d) \ \ \text{for all } \ \ d\in X,
\end{equation*}
where $\Omega:=\{x\in D\,|\, f(x)<f(\hat x)\}$.
\end{enumerate}
}
\end{remark}

\medskip

The following result gives a sufficient condition for a vector $w$ to be a second-order descent variation of a given $C^{1,1}$ function on $D$.
\begin{proposition}\label{prop2.2} Suppose that $f\in C^{1, 1}(D)$, $\hat x\in D$ and $d\in X$ satisfying $\langle f'(\hat x),  d\rangle \leq
0$. Denote
$$W_f=\bigg\{w\in X\,|\, \langle f'(\hat x),   w\rangle +\frac{1}2 \sup_{L\in\partial^2 f(\hat x)(d)}
\langle L,  d\rangle <0\bigg\}.$$ Then, $W_f$ is an open and convex set, and the  following inclusion  holds true
\begin{equation}\label{equ:1}
W_f\subseteq
W^2_{\delta}(f; \hat x, d).
\end{equation}
\end{proposition}
\begin{proof} Clearly, $W_f$ is an open and convex set. We now prove inclusion \eqref{equ:1}. The proof is indirect. Assume the opposite, i.e., there exists $\bar w\in W_f$ but $\bar w\notin W^2_{\delta}(f; \hat x, d)$. Then, for each $n\in\mathbb{N}$, there exist $\varepsilon_n\in (0, \frac{1}{n})$ and $w_n\in X$ with $\|w_n\|<\frac{1}{n}$ such that at least one of the following relations
\begin{align*}
& \hat x+ \varepsilon_n d +\varepsilon^2_n (\bar w + w_n)\in D,\\
&f(\hat x +\varepsilon_n d +\varepsilon^2_n(\bar w +w_n))< f(\hat x),
\end{align*} does not hold. For each $n\in\mathbb{N}$, put $x_n=\hat x+ \varepsilon_n d +\varepsilon^2_n (\bar w + w_n)$.  Clearly, the sequence  $\{x_n\}$  converges to $\hat x$ as $n\to\infty$. From the openness of $D$ it follows that $x_n\in D$ for all $n$ sufficient large. Thus, without loss of generality, we may assume that
 \begin{equation*}
f(\hat x +\varepsilon_n d +\varepsilon^2_n(\bar w +w_n))\geq  f(\hat x), \ \ \forall n\in\mathbb{N}.
\end{equation*}
This implies that
\begin{align}\label{b}
\langle f'(\hat x),  d\rangle &+\varepsilon_n\Big[ \frac{f(\hat x +
\varepsilon_n d +\varepsilon_n^2(\bar w + w_n))- f(\hat x + \varepsilon_n
d)}{\varepsilon_n^2}\Big] \notag\\
&\ \ \ \ \ \ \ \  +\varepsilon_n \Big[\frac{f(\hat x +\varepsilon_n d)- f(\hat x)-\varepsilon_n
\langle f'(\hat x),  d\rangle }{\varepsilon_n^2}\Big] \geq 0, \ \ \forall n\in\mathbb{N}.
\end{align}
By Theorem \ref{mean-th}, for each $n\in\mathbb{N}$, there exist $t_n\in(0,1)$  and
\begin{equation}\label{c}
L_n\in\partial^2 f(\hat x +t_n\varepsilon_n d)(\varepsilon_n d)=
\varepsilon_n\partial^2 f(\hat x +t_n\varepsilon_n d)(d)
\end{equation} such that
$$f(\hat x +\varepsilon_n d)- f(\hat x)- \varepsilon_n f'(\hat x; d)=\frac{1}2 \langle L_n, \varepsilon_n d\rangle, \ \ \forall n\in\mathbb{N}.$$
By \eqref{c}, $L_n=\varepsilon_n H_n$ for some $H_n\in
\partial^2 f(\hat x +t_n\varepsilon_n d)(d)$  and so
$$f(\hat x +\varepsilon_n d)- f(\hat x)- \varepsilon_n \langle f'(\hat x), d\rangle=\frac{1}2 \varepsilon_n^2\langle H_n,
d\rangle.
$$ It follows that
$$
\frac{f(\hat x +\varepsilon_n d)- f(\hat x)-\varepsilon_n f'(\hat x;
d)}{\varepsilon_n^2}= \frac{1}2 \langle H_n, d\rangle.
$$ Hence, by \eqref{b}, we have
\begin{equation}\label{d}
\langle f'(\hat x), d\rangle +\varepsilon_n\Big[ \frac{f(\hat x + \varepsilon_n d
+\varepsilon_n^2(\bar w + w_n))- f(\hat x + \varepsilon_n
d)}{\varepsilon_n^2}\Big] + \frac{1}{2}\varepsilon_n \langle H_n, d\rangle\geq 0.
\end{equation}  Since $\partial^2 f(\,\cdot\,)(d)$ is locally bounded near $\hat x$, we can
assume that $H_n$ converges weak$^*$ to $H_0$. By the upper semicontinuity of
the mapping $\partial^2 f(\,\cdot\,)$, we have $H_0\in
\partial^2 f(\hat x)(d)$. Besides, one has
$$\lim_{n\to\infty}\Big[\frac{f(\hat x + \varepsilon_n d
+\varepsilon_n^2(\bar w + w_n))- f(\hat x + \varepsilon_n
d)}{\varepsilon_n^2}\Big]= \langle f'(\hat x),  \bar w\rangle.
$$ Letting $n\to \infty$ in \eqref{d}  we obtain $\langle f'(\hat x), d\rangle\geq 0.$ Combining this with assumptions of the proposition,  we get $\langle f'(\hat x), d\rangle=0$. Substituting $\langle f'(\hat x), d\rangle=0$ into \eqref{d} and dividing
two sides by $\varepsilon_n>0$, we get
$$
\Big[ \frac{f(\hat x + \varepsilon_n d +\varepsilon_n^2(\bar w + w_n))-
f(\hat x + \varepsilon_n d)}{\varepsilon_n^2}\Big]+ \frac{1}{2}\langle H_n,
d\rangle\geq 0.
$$
Passing the limit, we obtain
$$
\langle f'(\hat x),  \bar w\rangle +\frac{1}2 \langle H_0,
d\rangle \geq 0,
$$
contrary to the fact that $\bar w\in W_f$. The proof is complete.
\end{proof}

The following result presents a characterization of the second-order tangent variation set to the null-set of a set-valued mapping between two general Banach spaces.
\begin{lemma} [{see \cite[Theorem 5]{Pales1}}]\label{lem3.4}
Assume that $H\colon D\to Y$ is strictly Fr\'echet differentiable
at $\hat x\in D$ such that $H'(\hat x)\colon X\to Y$ is surjective. Let $\Omega:=\{x\in X\,|\, H(x)=0\}$. Then $\bar
    w\in W^2_{\tau}(\Omega; \hat x, d)$ if and only if $\langle H'(\hat x), d\rangle=0$, $H$ is twice weakly directionally differentiable at $\hat x$ in the direction $d$ and the following condition holds:
    \begin{equation}\notag
    0\in \langle H'(\hat x), \bar w\rangle +\frac{1}2 H''(\hat x; d).
    \end{equation}
\end{lemma}

\begin{definition}{\rm  Let $C\subset X$ be a nonempty convex set and $x\in C$.
\begin{enumerate}[{\rm(i)}]
    \item  The {\em normal cone} to $C$ at $x$ is the set defined by
    \begin{equation*}
    N(C; x):=\{x^*\in X^*\;|\; \langle x^*, y-x\rangle \leq 0,\ \ \forall
    y\in C\}.
    \end{equation*}
    \item  The {\em adjoint set} of $C$ is the set defined by
    \begin{equation*}
    C^+:=\{\varphi: X\to \mathbb{R}\;|\;\varphi\ \text{is affine and}\  \varphi (x) \geq 0, \ \ \forall x \in C
    \}.
    \end{equation*}
\end{enumerate}
}
\end{definition}

Let $Q\subset X$ be a closed convex set with nonempty interior. The interior of $Q$ is denoted by $Q^\circ$. Let $\hat x\in
Q$ and $d\in X$. We define the following set:  
$$Q^\circ(\hat x, d):=\bigcup_{\bar\varepsilon>0}\bigcap_{\varepsilon <\bar\varepsilon \atop{
        \|w\|<\bar\varepsilon}}\Big[\frac{1}{\varepsilon^2}(Q -\hat x-\varepsilon d)+w\Big].
$$
This set plays an important role in the description of the second-order necessary optimality condition for problem $(P)$; see \cite{Pales1,Pales1994} for more details. It is easy to see that $Q^\circ(\hat x, d)$ is an open convex set and $Q^\circ(\hat x, d)= W^2_{\alpha}(Q; \hat x, d).$ The nonemptyness of  $Q^\circ(\hat x, d)$ is an important fact. As shown in \cite{Pales1}, it is necessary in order that   $d\in \overline{\rm cone}\, (Q-\hat x)$.

The following proposition follows directly  from \cite[Lemma 3 and Theorem 4]{Pales1}.

\begin{lemma}\label{lem3.5} Let $Q\subset X$ be a closed convex set with
    nonempty interior,  $\hat x\in Q$, and $d\in X$. Denote
    $C:={\rm cone}({\rm cone}(Q^\circ -\hat x)-d).$
    Then
    \begin{enumerate}[{\rm(i)}]
        \item  $\overline{Q^\circ(\hat x, d)} + C\subset Q^\circ(\hat x, d)$;
        \item  $Q^\circ(\hat x, d)\subset C$. If in addition, $d\in {\rm cone}(Q-\hat x)$, then the inclusion is equality and we therefore get $Q^\circ(\hat x, d)\neq\emptyset$;
        \item  Let $d\in \overline{\rm{cone}}\, (Q-\hat x)$ and $\phi(\,\cdot\,):=-\langle x^*, \,\cdot\,\rangle+t$ be an affine function defined on $X$, where $x^*\in X^*$, $t\in \mathbb{R}$. Then, the function $\phi$ is bounded from below on $C$ if and only if $x^*\in N(Q;\hat x)$ and $x^*(d)=0$. Moreover,
        $$C^+=\left\{\phi(\,\cdot\,)=-\langle x^*, \,\cdot\,\rangle+t\;|\; x^*\in N(Q;\hat x),\; \langle x^*, d\rangle=0,\;t\geq 0\right\}.$$
    \end{enumerate}
\end{lemma}

The following lemma is presented in \cite[Lemma 2]{Pales1} without
proof. Here we include the proof for completeness.

\begin{lemma}\label{lem3.2} Let $\gamma\colon X\to \mathbb{R}$ be a function which is convex and upper semicontinuous. Denote
    $$C:=\{x\in X\,|\, \gamma(x) <0\}.$$ Then, $C$ is open and convex. If
    $C$ is nonempty, then
    $$C^+=\{\varphi\colon X\to \mathbb{R}\;|\;\varphi\  \text{is affine and}\  \exists \mu\geq
    0: \varphi(x)+\mu\gamma(x)\geq 0,\ \ \forall x\in X\}.$$
\end{lemma}
\begin{proof} The openness and the convexity of $C$ are immediate from the upper semicontinuity and the convexity  of $\gamma$. Let $\varphi$ be an affine function defined on $X$. It is easily seen that $\varphi\in C^+$ if and only if the following convex system
\begin{equation*}
\begin{cases}
\varphi(x)&<0,
\\
\gamma(x)&<0,
\end{cases}
\end{equation*}
has no solution $x\in X$. By Ky Fan's Theorem \cite[Theorem 1]{Fan1957}, the inconsistency of the above system is equivalent to that there exist $\lambda\geq 0$, $\mu\geq 0$, not all zero, such that
    $$\lambda\varphi(x)+\mu\gamma(x)\geq 0,\ \ \forall  x\in X.$$
Under the assumption $C\neq\emptyset$, we can choose $\lambda=1$, and so the lemma follows.
\end{proof}

\begin{lemma} [{see \cite[Lemma 1]{Pales1}}]\label{lem3.3} Let $N$ be a positive integer and $C_1,  C_2, \ldots, C_N$ be
    nonempty convex sets in $X$ such that $ C_1, \ldots, C_{N-1}$  are open.
    Then $$\bigcap_{i=1}^N C_i=\emptyset$$ if and only if there exist
    affine functions $\varphi_1, \varphi_2, \ldots, \varphi_N\colon X\to \mathbb{R}$ not
    all constant such that $$ \sum_{i=1}^N \varphi_i =0, \ \ \varphi_i|_{C_i} \geq 0, \ \ \forall i=1, 2, \ldots, N.$$
\end{lemma}

\medskip

Define the {\em support function} of a nonempty set $C\subset X$ associated with $x^*\in X^*$ by
\begin{equation*}
\delta^*(x^*; C):=
\begin{cases}
\sup\{\langle x^*, c\rangle\,|\, c\in C\}, &\ \ \text{if}\ \ C\neq\emptyset,
\\
-\infty, &\ \ \text{if}\ \ C=\emptyset.
\end{cases}
\end{equation*}

\begin{lemma}[{see \cite[Lemma 4]{Pales1}}]\label{lem3.6} Let $X$ and $Y$ be  Banach spaces. Let $A\colon X\to Y$ be a bounded linear operator that maps $X$ onto $Y$ and let $K\subset Y$ be a nonempty convex set. Denote $C:=\{x\in X\,|\, Ax\in K\}.$    Then,
    $$C^+=\{\varphi\colon X\to \mathbb{R}\,|\, \varphi \mbox{ is affine and }\exists y^*\in Y^*\,:\,\varphi(x)\geq -\langle y^*, Ax\rangle+\delta^*(y^*;K),\;x\in X\}.$$
\end{lemma}
\section{Optimal conditions}
\label{Sect3}
We now return to problem $(P)$.
Put $J:=\{1, \ldots, m\}$. Hereafter, we use the notation $Q_{1}:=\{x\in D\,|\, G(x)\in Q\}$ and $Q_{2}:=\{x\in D\,|\, H(x)=0\}$.   
\begin{definition}{\rm  We say that $\hat x\in D\cap Q_1\cap Q_2$ is a {\em  weak Pareto efficient solution}  of $(P)$ if  there is no  $x\in D\cap Q_1\cap Q_2$  such that
    $$F_j(x)-F_j(\hat x)<0, \ \ \forall j=1, \ldots, m.$$
}
\end{definition}
The following lemma gives a necessary condition for a  weak Pareto efficient solution of $(P)$, which will be needed in the sequel. The idea of the proof is from \cite{Ben-Tal}.
\begin{lemma}\label{lem3.1} If $\hat x$ is a  weak Pareto efficient solution  of $(P)$, then
$$\bigg(\bigcap_{j=1}^{m} W_\delta^2(f_j; \hat x, d)\bigg)\cap W_\alpha^2(Q_{1}; \hat x,
d) \cap W^2_{\tau}(Q_{2};\hat x, d)=\emptyset.
$$
\end{lemma}
\begin{proof} Arguing by contradiction, assume that there exists $\bar w$ in the  above intersection. Then, there is $\bar\varepsilon>0$ such that
\begin{align*}
\hat x+\varepsilon d+\varepsilon^2(\bar w+w)&\in D,
\\
f_j(\hat x+\varepsilon d+\varepsilon^2(\bar w+w))&<f_j(\hat x), \ \ j\in J,
\\
\hat x+\varepsilon d+\varepsilon^2(\bar w+w) &\in Q_{1},
\end{align*}
hold for all $\|w\|<\bar\varepsilon$ and $0<\varepsilon<\bar\varepsilon$.  Furthermore, since $\bar w\in W^2_{\tau}(Q_{2}; \hat x, d)$, it follows that there exist sequences $\varepsilon_n>0$, $w_n\in Z$ converging to zero such that
$$\hat x+\varepsilon_n d+\varepsilon_n^2(\bar w+w_n)\in Q_{2}, \ \ \forall n\in\mathbb{N}.$$
Now choose $n_0$ large enough such that $\varepsilon_n<\bar \varepsilon$ and $\|w_n\|<\bar \varepsilon$ for all $n>n_0$. Then the sequence $x_n:=\hat x+\varepsilon_n d+\varepsilon_n^2(\bar w+w_n)$ converges to $\hat x$ and
\begin{equation*}
x_n\in D\cap Q_1\cap Q_2\ \ \text{and}\ \ f_j(x_n)<f_j(\hat x), \ \ \forall n>n_0,
\end{equation*}
which contradicts the optimality of $\hat x$.
\end{proof}

We say that $d\in X$ is a {\em critical direction} of $(P)$ at $\hat x$ if
\begin{equation}
\begin{cases} \langle f_j'(\hat x), d\rangle \leq 0, \ \ j\in J, \notag\\
\langle H'(\hat x), d\rangle=0, \notag\\
 G'(\hat x)d\in \overline{\text{cone}}(Q- G(\hat x)).
 \end{cases}
\end{equation}
The set of all critical direction of  $(P)$ at $\hat x$ is denoted by $\mathcal{C}(\hat x)$.
A direction $d$ is called {\it a regular direction} at $\hat x$
if $H''(\hat x; d)\neq \emptyset$, $G''(\hat x, d)\neq \emptyset$ and $Q^\circ(G(\hat x), G'(\hat x)d)\neq
\emptyset$.

\medskip
We now state the main result of the paper.
\begin{theorem}\label{main-result-1}  Assume that $\hat x$ is a  weak Pareto efficient solution of  $({P})$,  $f_j \in C^{1,1}(D)$, $j\in J$, $H$ and $G$
	are strictly  differentiable at $\hat x$ such that $H'(\hat x)(X)$
	is a closed subspace of $Y$. Then, for  all  critical
	directions $d\in {\mathcal{C}}(\hat x)$ and convex sets $K\subset H''(\hat
	x; d)$ and $M\subset G''(\hat x; d)$,  there exist nonnegative
	numbers $\mu_1, \ldots, \mu_m,$  and functionals $y^*\in Y^*$, $z^*\in Z^*$, not all zero such that the following
	conditions hold:
	\begin{enumerate}[{\rm(i)}]
		\item  the complementarity conditions
		\begin{equation*}
		z^*\in N(Q; G(\hat x))\quad\text{and}\quad \langle z^*, G'(\hat
		x)d\rangle=0.
		\end{equation*}
		\item  the first-order necessary condition
		\begin{equation*}
		\sum_{j=1}^m \mu_j f'_j(\hat x) +   H'(\hat x)^* y^* + G'^*(\hat x)z^*=0,
		\end{equation*}
		\item  the second-order necessary condition
		\begin{equation*}
		\sum_{j=1}^m \mu_j\sup_{L\in\partial^2 f_j(\hat x)(d)}\langle L, d\rangle  \geq\delta^*(-y^*; K) +\delta^*(-z^*; M) + 2\delta^*(
		z^*; Q^\circ(G(\hat x), G'(\hat x)d)). 
		\end{equation*}
	\end{enumerate}
\end{theorem}
We first prove this theorem for the case that $G(x)=x$ for all $x\in D$, i.e., $Q_{1}=D\cap Q$. In this case, problem $(P)$ is denoted by $(P_1)$ and the obtained result is as follows.
\begin{theorem}\label{main-result} Assume that $\hat x$ is a  weak Pareto efficient  solution of  $(P_1)$, $f_j \in C^{1,1}(D)$ for all $j\in J$, $H$ is
strictly  differentiable at $\hat x$ such that $H'(\hat x)(X)$ is a closed subspace of $Y$. Let  $d$ be a critical direction of $(P_1)$ at $\hat x$. Assume that $K$ is a convex subset in $H''(\hat x; d)$. Then, there exists $(\mu, x^*, y^*)\in (\mathbb{R}^m_+\times  X^*\times Y^*)\setminus\{0\}$ such that the following conditions hold:
\begin{enumerate}[{\rm(i)}] 
    \item the complementarity conditions
    \begin{equation}\label{complementary-condition}
    x^*\in N(Q; \hat x),\ \ \text{and}\quad \langle x^*, d\rangle=0,
    \end{equation}

    \item  the first-order necessary condition
    \begin{equation}\label{first-order-condition}
    \sum_{j=1}^m\mu_j f_j'(\hat x)  + H'(\hat
    x)^* y^* + x^*=0,
    \end{equation}

    \item  the second-order necessary condition
    \begin{equation}\label{second-order-condition}
        \sum_{j=1}^m\mu_j \sup_{L\in\partial^2 f_j(\hat x)(d)}\langle L, d\rangle \geq  \delta^*(-y^*; K) +  2\delta^*( x^*; Q^\circ(\hat x, d)).
    \end{equation}
\end{enumerate}
\end{theorem}
\begin{proof} We first prove the theorem when $d$ is a regular direction of $(P_1)$ at $\hat x$. Let us consider the following possible cases.
\medskip

{\bf Case 1.} There exists $j_0\in J$ such that $W_{f_{j_0}}=\emptyset$, where
$$W_{f_{j_0}}:=\bigg\{w\in X\,|\, \langle f'_{j_0}(\hat x),  w\rangle  +\frac{1}2 \sup_{L\in\partial^2 f_{j_0}(\hat x)(d)}
\langle L,  d\rangle <0\bigg\}.$$

Then, for all $w\in X$, we have
\begin{equation}\label{20}
 \langle f'_{j_0}(\hat x), w\rangle  +\frac{1}2 \sup_{L\in\partial^2 f_{j_0}(\hat x)(d)}\langle L, d\rangle \geq  0.
 \end{equation} We choose $\mu_{j_0} =1$,  $\mu_j=0$ for all $j\in J\backslash\{j_0\}$,  $y^*=0$ and $x^*=0$. Fixing any $x\in X$ and substituting $w=t x$ with $ t >0$ into \eqref{20} and then dividing two sides by $t$, we have
\begin{equation*}
 \langle f'_{j_0}(\hat x), x\rangle  + \frac{1}{2t} \sup_{L\in\partial^2 f_{j_0}(\hat x)(d)}\langle L, d\rangle \geq  0.
 \end{equation*}
Letting $t\to +\infty$, we get $\langle f'_{j_0}(\hat x), x\rangle=0$ for all $x\in X$. This implies that $f'_{j_0}(\hat x)=0$. Substituting
$f'_{j_0}(\hat x)=0$ into \eqref{20}, we have
$$\sup_{L\in\partial^2f_{j_0}(\hat x)(d)}\langle L, d\rangle \geq  0.
$$ Hence we obtain the conclusions of the theorem.
\medskip

{\bf Case 2.} $H'(\hat x)(X)$ is a proper subspace of $Y$. Since $H'(\hat x)(X)$ is closed, by the separation theorem, there exists $y_0^*\in Y^*\setminus\{0\}$ such that $y^*_0$ is identically zero on the
range of $H'(\hat x)$. To obtain the desired conclusions, we take $\mu_j=0$ for all $j\in J$,  $x^*=0$, and $y^*= y^*_0$ or
$y^*=-y^*_0$.
\medskip

{\bf Case 3.} $W_{f_{j}}\neq \emptyset$ for all $j\in J$ and $H'(\hat x)(X)= Y.$

Put
$$W_H=\Big\{ w\in X\,|\,  H'(\hat x) w \in - \frac{1}2 K\Big\}$$ and
$$W_Q=W_\alpha^2(Q_{1}; \hat  x, d)= Q^\circ(\hat x, d).$$ It is clear
that $W_H$ and $W_Q$ are nonempty and convex sets. Moreover, $W_Q$ is open. Thus,  the sets $W_{f_j}$, $j\in J$, $W_Q$, and $W_H$  are nonempty. Furthermore, they are convex and the first $m+1$ sets are open. By Proposition~\ref{prop2.2} and Lemma~\ref{lem3.4}, we have
\begin{equation*}
W_{f_j}\subseteq W^2_\delta(f_j; \hat x, d),  \quad W_H \subseteq W^2_\tau(Q_{2}; \hat x,
d).
\end{equation*}
From Lemma~\ref{lem3.1} it follows that
$$\bigg(\bigcap_{j=1}^{m} W_\delta^2(f_j; \hat x, d)\bigg) \cap W_\alpha^2(Q_{1}; \hat x,
d)\cap W^2_{\tau}(Q_{2};\hat x, d)=\emptyset.
$$ Hence,
\begin{equation*}
\bigg(\bigcap_{j=1}^{m} W_{f_j}\bigg) \cap W_Q \cap
W_H=\emptyset.
\end{equation*}
It follows from Lemma~\ref{lem3.3} that there exist affine functions $\varphi_j$,  $j\in J$,  $\phi_Q$ and $\phi_H$, not all constant, such that
$$\varphi_j\in W_{f_j}^+, j\in J,  
\phi_Q\in W_Q^+, \phi_H\in W_H ^+
$$ and
\begin{equation}\label{24}
\sum_{j=0}^{m}\varphi_j+ \phi_Q+\phi_H = 0.
\end{equation}
By  Lemma~\ref{lem3.2}, there exist nonnegative numbers $\mu_j$  such that
\begin{equation}\label{25}
\varphi_j(x) +\mu_j\Big( \langle f'_j(\hat x), x\rangle +
\frac{1}2 \sup_{L\in\partial^2 f_j(\hat x)(d)}\langle L,
d\rangle\Big)\geq 0, \quad \forall x\in X, j\in J.
\end{equation}
 
Assume that $\phi_Q(\cdot)=-\langle x^*,\,\cdot\,\rangle+t$, where $x^*\in X^*$ and $t\in\mathbb{R}$. Since $\phi_Q\in W_Q^+$, we have $\phi_Q(x)\geq 0$ for all $x\in W_Q=Q^\circ(\hat x, d)$. From this and Lemma~\ref{lem3.5}(i) it follows that
$$t\geq \delta^*(x^*; Q^\circ( \hat x, d))\ \ \text{and} \ \ -\langle x^*, u+x\rangle+t\geq 0$$
for all $u\in \overline{Q^\circ(\hat x, d)}$ and $x\in C$. Fix $u_0\in \overline{Q^\circ(\hat x, d)}$, then we have
$$-\langle x^*, x\rangle+t\geq \langle x^*, u_0\rangle,\ \ \forall x\in C.$$
Consequently, $\phi_Q$ is bounded from below on $C$. By Lemma~\ref{lem3.5}(iii), $x^*\in N(Q;\hat x)$ and $\langle x^*, d\rangle=0$.
Clearly,
 \begin{equation}\label{27n}
 \phi_Q(x) = -\langle x^*, x\rangle  + t\geq  -\langle x^*, x\rangle   + \delta^*(x^*; Q^\circ( \hat x,
 d)), \ \ \forall x\in X.
 \end{equation}
By Lemma~\ref{lem3.6}, there exists $y^*\in Y^*$ such that
\begin{equation}\label{29}
\phi_{H}(x) + \langle y^* H'(\hat x), x\rangle \geq \frac{1}{2}  
\delta^*(-y^*; K),\ \ \forall x\in X.
\end{equation}
Adding inequalities \eqref{25}--\eqref{29} and using \eqref{24}, we
obtain
\begin{align}\label{27}
\sum_{j=1}^{m}\mu_j\Big(\langle f'_j(\hat x), x\rangle + \frac{1}2
\sup_{L\in\partial^2 f_j(\hat x)(d)}\langle L, d\rangle \Big)&+ \langle H'(\hat x)^*y^*, x\rangle+\langle x^*, x\rangle\notag\\
&\geq \frac{1}{2}  
\delta^*(-y^*; K)+\delta^*(x^*, Q^\circ(\hat x, d))
\end{align}
for all $x\in X$. Fixing any $z\in X$ and substituting $x=t z$, where $t>0$, into \eqref{27} and dividing two side by $t$, we
obtain
\begin{align}\label{28}
\sum_{j=1}^{m}\mu_j \langle f'_j(\hat x), z\rangle& + \langle H'(\hat x)^*y^*, z\rangle +\langle x^*, z\rangle \notag 
\\
&\geq -\frac{1}{2t}
\sum_{j=1}^{m}\mu_j\sup_{L\in\partial^2 f_j(\hat x)(d)}\langle L, d\rangle+\frac{1}{2t}\Big(\delta^*(-y^*; K) +
2\delta^*(x^*, Q^\circ(\hat x, d))\Big).
\end{align}
Letting $t\to+\infty$ in \eqref{28}, we have
$$\sum_{j=1}^{m}\mu_j \langle f'_j(\hat x), z\rangle + \langle H'(\hat x)^*y^*, z\rangle  +
\langle x^*, z\rangle \geq 0, \ \ \forall z\in X.$$
It follows that
\begin{equation}\label{equ:first-order}
\sum_{j=1}^{m}\mu_jf'_j(\hat x)  +
H'(\hat x)^*y^* + x^*=0.
\end{equation}
Substituting this into \eqref{28}, we have
$$\sum_{j=1}^{m}\mu_j\sup_{L\in\partial^2
    f_j(\hat x)(d)}\langle L, d\rangle  \geq  \delta^*(-y^*; K) +2\delta^*(x^*, Q^\circ(\hat x, d)).
$$
Then, $(\mu_1, \ldots, \mu_m, \lambda_k, x^*, y^*)$ satisfies all conditions \eqref{complementary-condition}--\eqref{second-order-condition}. In this case, we claim that $(\mu_1, \ldots, \mu_m,  x^*)\neq 0$. Indeed, we first show that at least one of multipliers  $\mu_1, \ldots, \mu_m,$  $x^*, y^*$ is different from zero. If otherwise, then, since \eqref{25}--\eqref{29}, $\varphi_j, j\in J$, $\phi_Q$ and $\phi_H$ must be all constants, a contradiction. Thus, if $(\mu_1, \ldots, \mu_m,  x^*)= 0$, then $y^*\neq 0$. Substituting this into \eqref{equ:first-order} we have $H'(\hat x)^*y^*=0$. Since $H'(\hat x)X=Y$, we have $y^*=0$, contrary to the fact that $y^*\neq 0$.

We now consider the case that $d$ is a nonregular critical direction of $(P_1)$ at $\hat x$. Clearly, $\hat d=0$ is a critical direction at $\hat x$. Moreover, it is easy to check that $H''(\hat x; 0)=\{0\}$ and
$$Q^\circ(\hat x; 0)=\text{cone}\, (Q^\circ-\hat x)\neq\emptyset.$$
Thus $\hat d=0$ is also a regular  direction of  $(P)$ at $\hat x$. Now, apply the above proof to $\hat d=0$ and $K=\{0\}$, there exist nonnegative numbers $\mu_1, \ldots, \mu_m$ and functionals $x^*\in N(Q; \hat x)$, $y^*\in Y^*$  not all zero  satisfying condition \eqref{first-order-condition}. The nonregularity of $d$ means that either $H''(\hat x; d)$, or $Q^\circ(\hat x; d)$ is  empty. Thus the left-hand side of \eqref{second-order-condition} equals positive infinity and condition \eqref{second-order-condition} is trivial.  The proof is complete.
\end{proof}
\medskip

\noindent{\em Proof of Theorem \ref{main-result-1}.} By introducing a new variable $z\in Z$, we can reduce problem $(P)$ to the following problem:
	\begin{equation}\notag
	(\widetilde P)\quad
	\begin{cases}
	\text{Min}\,_{\mathbb{R}^m_+}\widetilde F(x, z):=F(x)\\
	\text{subject to}\\
	(x, z)\in X\times Q,\\
	\widetilde{H}(x, z):= (H(x), G(x)-z)=(0, 0).
	\end{cases}
	\end{equation}
	Notice that $\widetilde{H}'(\hat x, \hat z)(X\times Z)=H'(\hat x)X\times (
	G'(\hat x)X-Z)=H'(\hat x)X\times Z$ is a closed subspace in $
	Y\times Z$. By Theorem \ref{main-result}, we can find multipliers which satisfy the desired conclusion of the theorem. $\hfill\Box$

\begin{remark}{\rm It is worth noting that Theorem \ref{main-result-1}  embraces a first-order condition as a special case. Indeed, let $\hat x$ be a  weak Pareto efficient  solution of  $(P)$, and denote by $\Lambda (\hat x)$ the set of Lagrange multipliers  $(\mu, x^*, y^*)\in (\mathbb{R}^m_+\times X^*\times Y^*)\setminus\{0\}$
 which satisfy conditions (i)--(ii) of Theorem \ref{main-result-1}. Applying Theorem \ref{main-result-1} for $d=0$,  the set of Lagrange multipliers $\Lambda (\hat x)$ is always nonempty.
}
\end{remark}

We finish this section by presenting a corollary of Theorem \ref{main-result} for the case that $H=(h_1, \ldots, h_p)$, $G=(g_1, \ldots, g_k)$, and $Q=-\mathbb{R}^k_+$. The obtained result generalizes \cite[Theorem 2.4]{Pando} to the multiobjective optimization case.
\begin{corollary} Consider problem $(P_1)$ where $Q=-\mathbb{R}^k_+$ and  $H=(h_1, \ldots, h_p)$,  $G=(g_1, \ldots, g_k)$ are vector-valued functions with $C^{1,1}(D)$ components. Assume that $\hat x$ is a  weak Pareto efficient  solution of  $(P_1)$. Then, for every critical direction $d$, there exist $(\mu, \lambda, \beta) \in (\mathbb{R}^m_+\times\mathbb{R}^k_+\times\mathbb{R}^p)\setminus\{0\},$    $L_i\in\partial^2f_j(\hat x)(d)$, $j\in J$, and  $M_i\in\partial^2g_i(\hat x)(d)$, $i=1, \ldots, k$, $K_l\in\partial^2h_l(\hat x)(d)$, $l=1, \ldots, p$ such that
the following conditions hold:
\begin{enumerate}[{\rm(i)}] 
    \item the complementarity conditions
    \begin{equation*}
    \lambda_ig_i(\hat x)=0,\ \  i=1, \ldots, k,
    \end{equation*}

    \item  the first-order necessary condition
    \begin{equation*}
    \sum_{j=1}^m\mu_j f_j'(\hat x) + \sum_{i=1}^k\lambda_i g'_i(\hat x) + \sum_{l=1}^p\beta_l h'_l(\hat x)=0,
    \end{equation*}

    \item  the second-order necessary condition
    \begin{equation*}
    \sum_{j=1}^m\mu_j \langle L_j, d\rangle+
    \sum_{i=1}^k \lambda_i\langle M_i,
    d\rangle  +\sum_{l=1}^p \beta_l\langle K_l,
    d\rangle\geq 0.
    \end{equation*}
\end{enumerate}
\end{corollary}
\begin{proof} By Corollary \ref{corollary27}, $H''(\hat x; d)$ and $G''(\hat x; d)$ are nonempty. Thus, every critical direction at $\hat x$  is also regular.  Moreover, we have
\begin{align*}
H''(\hat x; d)&\subset  \partial^2 h_1(\hat x)(d)(d) \times\ldots\times \partial^2 h_p(\hat x)(d)(d),
\\
G''(\hat x; d)&\subset \partial^2 g_1(\hat x)(d)(d) \times\ldots\times \partial^2 g_k(\hat x)(d)(d).
\end{align*}
Hence, there exist $K_l\in \partial^2h_l(\hat x)(d)$, $l=1, \ldots, p$, and $M_i\in \partial^2g_i(\hat x)(d)$, $i=1, \ldots, k$, such that 
\begin{align*} 
&\left(\langle K_1, d\rangle, \ldots,  \langle K_p, d\rangle\right)\in H''(\hat x; d),
\\
&\left(\langle M_1, d\rangle, \ldots,  \langle M_k, d\rangle\right)\in G''(\hat x; d).
\end{align*}
Applying Theorem \ref{main-result-1} for $d\in\mathcal{C}(\hat x)$, $Q=-\mathbb{R}^k_+$, $K=\left\{\left(\langle K_1, d\rangle, \ldots,  \langle K_p, d\rangle\right)\right\}$, and $M=\left\{\left(\langle M_1, d\rangle, \ldots,  \langle M_k, d\rangle\right)\right\}$, there exist  multipliers which satisfy the desired conclusion of the corollary.
\end{proof}


\end{document}